\documentclass[a4paper,11pt]{amsart}

\usepackage{graphicx}
\usepackage{amsmath,amssymb,amsthm,amsfonts}
\usepackage{amssymb}

\usepackage{dsfont}
\usepackage{color, framed}

\usepackage[colorlinks]{hyperref}
\hypersetup{linkcolor=blue,citecolor=blue,filecolor=black,urlcolor=blue}

\newtheorem{thm}{Theorem}[section]
\newtheorem{theorem}[thm]{Theorem}
\newtheorem{cor}[thm]{Corollary}
\newtheorem{lem}[thm]{Lemma}
\newtheorem{prop}[thm]{Proposition}
\newtheorem{proposition}[thm]{Proposition}
\theoremstyle{definition}
\newtheorem{rem}[thm]{Remark}
\newtheorem{remark}[thm]{Remark}

\numberwithin{equation}{section}

\newcommand{\bremark}{\begin{rem} \textup}
\newcommand{\eremark}{\end{rem} }

\newcommand{\cuad}{{\sqcap\kern-.68em\sqcup}}

\newcommand{\la}{\lambda}

\newcommand{\R}{{\mathbb{R}}}

\renewcommand{\rho}{\varrho}
\renewcommand{\theta}{\vartheta}
\newcommand{\pa}{\partial}
\newcommand{\eps}{\varepsilon}
\newcommand{\s}{\sigma}
\renewcommand{\S}{\mathbb{S}}

\begin{document}


\parindent 0pc
\parskip 6pt
\overfullrule=0pt

\title[Monotonicity and rigidity for elliptic systems]{Monotonicity and rigidity of solutions to some elliptic systems with uniform limits}

\author {Alberto Farina, Berardino Sciunzi and Nicola Soave}

%\date{\today}

\address{
\hbox{\parbox{5.7in}{\medskip\noindent
Alberto Farina \\
LAMFA, CNRS UMR 7352, Universit\'e de Picardie Jules Verne,\\
33 rue Saint-Leu, 80039 Amiens (France) \\[2pt]
{\em{E-mail address: }}{\tt alberto.farina@u-picardie.fr} \\[5 pt]
Berardino Sciunzi\\
Dipartimento di Matematica e Informatica, UNICAL,
Ponte Pietro Bucci 31B, 87036 Arcavacata di Rende, Cosenza (Italy),\\[2pt]
{\em{E-mail address: }}{\tt sciunzi@mat.unical.it} \\ [5pt]
Nicola Soave\\
Dipartimento di Matematica, Politecnico di Milano, \\
Piazza Leonardo da Vinci, 32, 20133 Milano (Italy) \\[2pt]
{\em{E-mail address: }}{\tt nicola.soave@gmail.com, nicola.soave@polimi.it}}}}



\thanks{\em{Acknowledgements:} Alberto Farina and Berardino Sciunzi are partially supported by ERC grant n. 277749 ``Elliptic Pde's and Symmetry of Interfaces and Layers for Odd Nonlinearities - EPSILON". Alberto Farina and Nicola Soave are partially supported through the project ERC grant 2013 n. 339958 ``Complex Patterns for Strongly Interacting Dynamical Systems - COMPAT''. Nicola Soave is partially supported through the project PRIN-2015KB9WPT\texttt{\char`_}010 Grant: ``Variational methods, with applications to problems in mathematical physics and geometry", and by the INDAM-GNAMPA project ``Aspetti non-locali in fenomeni di segregazione".}

\begin{abstract}
In this paper we prove the validity of Gibbons' conjecture for a coupled competing Gross-Pitaevskii system. We also provide sharp a priori bounds, regularity results and additional Liouville-type theorems.
\end{abstract}

\maketitle

\date{\today}

%\date{}

%END TOPMATTER

\maketitle

\section{introduction}
This paper concerns the study of the qualitative properties, with particular emphasis to symmetry and monotonicity, of non-negative solutions to the elliptic system
\begin{equation}\tag{$P$}\label{system}
\begin{cases}
-\Delta u \, =u-u^3-\Lambda uv^2& \text{in}\quad\mathbb{R}^N  \\
-\Delta v \, =v-v^3-\Lambda u^2v& \text{in}\quad\mathbb{R}^N  \\
\quad u,v \ge 0 &  \text{in}\quad\mathbb{R}^N, 
\end{cases}\qquad \text{with $\Lambda > 0$.}
\end{equation}
As discussed in \cite{AftSou,AlaBro1,DroMal,Mal} (see also the references therein), the problem under investigation with $\Lambda>1$ arises in the study of domain walls and interface layers of two-components Bose-Einstein condensates in the segregation regime. Motivated by the physical interpretation, in \cite{AlaBro1} it is proved that \eqref{system} in dimension $N=1$ has a solution satisfying the conditions
\begin{equation}\label{limit 1 d}
\begin{split}
&u(t) \to 1 \qquad v(t) \to 0 \qquad \text{as $t \to +\infty$} \\
&u(t) \to 0 \qquad v(t) \to 1 \qquad \text{as $t \to -\infty$}, \\
\end{split}
\end{equation}
that is, an heteroclinic connection between the equilibria $(0,1)$ and $(1,0)$. This solution satisfies the monotonicity condition
\begin{equation}\label{monot prop 1d}
u'>0 \quad \text{and} \quad v'<0 \quad \text{in $\R$},
\end{equation}
is spectrally and nonlinearly stable (see \cite{AlaBro1}), and is unique (modulo translation) within the class of solutions with a monotone component (see \cite{AftSou}). Furthermore, in \cite{AftSou} precise asymptotic estimates in the limit $\Lambda \to +\infty$ are provided. 

By the study of the $1$-dimensional problem, and by the results in \cite{AftSou}, it emerges a deep connection between problem \eqref{system} and 
\begin{equation}\label{syst solito}
\begin{cases}
\Delta U \, =U V^2& \text{in}\quad\mathbb{R}^N  \\
\Delta V \, =U^2 V& \text{in}\quad\mathbb{R}^N  \\
U, V \ge 0 &  \text{in}\quad\mathbb{R}^N.
\end{cases}
\end{equation}
Firstly, also \eqref{syst solito} in $\R$ has a unique stable solution (modulo translations and scalings) which satisfies the same monotonicity property as in \eqref{monot prop 1d}, and has limits
\[
\begin{split}
&U(t) \to +\infty \qquad V(t) \to 0 \qquad \text{as $t \to +\infty$} \\
&U(t) \to 0 \qquad V(t) \to +\infty \qquad \text{as $t \to -\infty$},
\end{split}
\]
see \cite{BeLiWeZh, BeTeWaWe}. Secondly, it is shown in \cite{AftSou,SoZiP} that such $1$-dimensional solution appears as limit in a suitable blow-up analysis near the regular part of the interface of solutions to \eqref{system} in the \emph{strong competition regime} $\Lambda \to +\infty$.

Problem \eqref{syst solito} has been intensively studied in the last years also in higher dimension, starting from the seminal paper \cite{BeLiWeZh}. We refer to \cite{BeTeWaWe, SoZi1, SoZi2} for existence, and to \cite{BeTeWaWe, FarSyst, FaSo, NoTaTeVe, SoTe, Wa1, Wa2} for classification results. In particular, in \cite{FaSo} the authors proved that in any dimension $N \ge 2$, a solution to \eqref{syst solito} with at most polynomial growth at infinity and satisfying
\[
\begin{split}
&U(x',x_N) \to +\infty \qquad V(x',x_N) \to 0 \qquad \text{as $x_N \to +\infty$} \\
&U(x',x_N) \to 0 \qquad V(x',x_N) \to +\infty \qquad \text{as $x_N \to -\infty$}, \\
\end{split}
\]
the limit being uniform in $x' \in \R^{N-1}$, depends only on $x_N$. This gives an affirmative answer to a conjecture raised in \cite{BeLiWeZh}, which is the natural counterpart of the famous \emph{Gibbons' conjecture} for the Allen-Cahn equation, for which we refer to \cite{BaBaGu, BeHaMo, Fa99}.

Motivated by the $1$-dimensional analysis carried on in \cite{AftSou, AlaBro1}, and 
inspired by the results in \cite{FaSo}, in this paper we address the following issue: is it true that, in any dimension $N \ge 1$, any solution to \eqref{system} satisfying the condition \eqref{limit 1 d} as $x_N \to \pm\infty$, uniformly in the other variables, depends only on $x_N$? The first of our main results is the positive answer to this question.

\begin{theorem}\label{thm: main 1}
Let $N \ge 1$, $\Lambda >1$, and let $(u,v) \in L^3_{loc}(\R^N) \times  L^3_{loc}(\R^N)$ be a distributional solution of \eqref{system}, satisfying the assumption
\begin{equation}\tag{$h_\infty$}\label{uniform limit}
\begin{split}
&u(x',x_N) \to 1 \qquad v(x',x_N) \to 0 \qquad \text{as $x_N \to +\infty$} \\
&u(x',x_N) \to 0 \qquad v(x',x_N) \to 1 \qquad \text{as $x_N \to -\infty$}, 
\end{split}
\end{equation}
uniformly with respect to $x'\in \mathbb{R}^{N-1}$. Then $(u,v)$ is smooth, depends only on $x_N$, and
\[
\pa_{N} u >0, \quad \pa_N v <0 \quad \text{in $\R^N$}.
\]
\end{theorem}
Here and in the rest of the paper, we use the common notation $x=(x',x_N) \in \R^{N-1} \times \R$ for points of $\R^N$, and we denote by $\pa_N$ the partial derivative with respect to $x_N$. In general, in what follows we always write $\pa_\nu w$ to denote the directional derivative of a function $w$ in a direction $\nu \in \mathbb{S}^{N-1}$. 

As byproduct of Theorem \ref{thm: main 1}, we obtain the uniqueness (modulo translations) of the positive 1D heteroclinic connections between $(0,1)$ and $(1,0)$, without any additional assumption on $(u,v)$. For minimal solutions (in a suitable sense), this result was conjectured in \cite[Section 5]{AlaBro1}, and proved in \cite{AftSou} as a consequence of \cite[Theorem 1.3]{AftSou}, where it is showed the uniqueness of positive 1D heteroclinic connections with one monotone component. Here we can remove the monotonicity assumption and extend the uniqueness in any dimension.

\begin{cor}\label{cor: uniqueness}
For $N \ge 1$ and $\Lambda >1$, there exists exactly one distributional solution $(u,v) \in L^3_{loc}(\R^N) \times  L^3_{loc}(\R^N)$ of \eqref{system} satisfying \eqref{uniform limit}, modulo translations.
\end{cor}

The first step in the proof of Theorem \ref{thm: main 1} is the following universal estimate regarding any solution to \eqref{system} (not necessarily positive), with arbitrary $\Lambda  > 0$, which we think can be of independent interest.

\begin{thm}\label{universal upper bound}
Assume $ N\ge 1$ and let $(u,v) \in L^3_{loc}(\R^N) \times  L^3_{loc}(\R^N)$ be a distributional solution of 
\begin{equation}\label{systemUE}
\begin{cases}
-\Delta u \, =u-u^3-\Lambda uv^2& \text{in}\quad\mathbb{R}^N  \\
-\Delta v \, =v-v^3-\Lambda u^2v& \text{in}\quad\mathbb{R}^N.  
\end{cases}
\end{equation} 
It results that:
\begin{itemize}
\item [($i$)] If $ \Lambda > 0$, then $u$ and $v$ are smooth and satisfy
\[
\vert u \vert \le 1, \quad \vert v \vert \le 1 \qquad \text{in} \qquad \R^N.
\]
\item [($ii$)] If $ \Lambda \ge 1$, then
\[
u^2 + v^2  \le 1 \qquad \text{in} \qquad \R^N.
\]
\item [($iii$)] If $ \Lambda \in (0,1)$, then
\[
u^2 + v^2  \le  \frac{2}{\Lambda +1} \qquad \text{in} \qquad \R^N.
\]
\end{itemize}
\end{thm}

\begin{remark}
1) The previous estimates are sharp. To see this, it is sufficient to observe that the constant solutions $(\pm 1,0)$ or $(0,\pm 1)$ realize the equality in items ($i$) and ($ii$), and $(\frac{1}{\sqrt{\Lambda +1}}, \frac{1}{\sqrt{\Lambda +1}})$ realizes the equality in item ($iii$). 

2) Regarding item ($i$), we note that, whenever we consider a nonconstant solution, we immediately gain the strict inequalities $|u|,|v| <1$ by the strong maximum principle.

3) By elliptic estimates, we can also obtain universal bounds (depending only on the dimension $N$) for all the derivatives of any solution to \eqref{system}.
\end{remark}

We also observe that item ($ii$) permits to weaken the assumptions of Theorem \ref{thm: main 1} in the following way:

\begin{cor}\label{cor: sharp}
Let $N \ge 1$, $\Lambda >1$, and let $(u,v) \in L^3_{loc}(\R^N) \times  L^3_{loc}(\R^N)$ be a distributional solution of \eqref{system}, satisfying the assumption
\begin{equation}\label{uniform limit sharp}
\lim_{x_N \to \pm \infty} \Big( u(x',x_N)- v(x',x_N) \Big) = \pm 1,
\end{equation}
uniformly with respect to $x'\in \mathbb{R}^{N-1}$. Then $(u,v)$ is smooth, satisfies \eqref{uniform limit}, depends only on $x_N$, and
\[
\pa_{N} u >0, \quad \pa_N v <0 \quad \text{in $\R^N$}.
\]
\end{cor}
Clearly, this result gives also a stronger version of Corollary \ref{cor: uniqueness}, where assumption \eqref{uniform limit} is replaced by \eqref{uniform limit sharp}.

\bigskip

Let us focus now on the case $\Lambda \in (0,1)$. The existence of heteroclinic connection between ($0,1)$ and $(1,0)$ was not proved in this setting. We can show that this is natural, as specified by the following Liouville-type theorem. 

\begin{thm}\label{sopra<1}
Assume $ \Lambda \in (0,1)$, $N \ge 1$, and let $(u,v)$ be a solution of \eqref{system}:
\[
\begin{cases}
-\Delta u\,=u-u^3- \Lambda uv^2& \text{in}\quad\mathbb{R}^N  \\
-\Delta v\,=v-v^3- \Lambda u^2v& \text{in}\quad\mathbb{R}^N  \\
\quad u,v >0 & \text{in}\quad\mathbb{R}^N.
\end{cases}
\]
Then $ u = v = \frac{1} {{\sqrt {1+\Lambda}}}$.
\end{thm}

The case $\Lambda=1$ is not covered by Theorems \ref{thm: main 1} and \ref{sopra<1}. Also in such setting the existence of non-constant positive solutions is an open problem, and we can show that once again no such solution can exist, at least in dimension $N \le 2$. 

\begin{thm}\label{sopra1}
Assume $ N \le 2$, and let $(u,v)$ be a solution of \eqref{system} with $\Lambda=1$:
\[
\begin{cases}
-\Delta u\,=u-u^3- uv^2& \text{in}\quad\mathbb{R}^N  \\
-\Delta v\,=v-v^3- u^2v& \text{in}\quad\mathbb{R}^N  \\
\quad u,v  > 0 & \text{in}\quad\mathbb{R}^N.
\end{cases}
\]
Then $(u,v)$ is constant, necessarily satisfying $ u^2 +v^2 =1$.
\end{thm}
The case $\Lambda=1$ in higher dimension is open.

Further considerations are devoted to the special case $\Lambda=3$. In such a situation, system \eqref{system} together with the limit condition \eqref{uniform limit} in $\R$ has the explicit solutions 
\[
u(t) = \frac{1+ \tanh\left(\frac{t}{\sqrt{2}}\right)}{2}, \quad v(t) = \frac{ 1- \tanh\left(\frac{t}{\sqrt{2}}\right)}{2},
\]
as already observed in \cite{DroMal, Mal}. We can actually provide a complete classification of the heteroclinic connections between $(0,1)$ and $(1,0)$ in arbitrary dimension, \emph{without any sign-assumption on $(u,v)$}. 

\begin{theorem}\label{sopra3}
Assume $ N \ge 1$, and let $(u,v)$ be a solution of 
\begin{equation}\label{sistema3}
\begin{cases}
-\Delta u\,=u-u^3- 3 uv^2& \text{in}\quad\mathbb{R}^N  \\
-\Delta v\,=v-v^3- 3 u^2v& \text{in}\quad\mathbb{R}^N,  
\end{cases}
\end{equation}
such that
\begin{equation}\tag{$h_\infty$}\label{uniform limit}
\begin{split}
&u(x',x_N) \to 1 \qquad v(x',x_N) \to 0 \qquad \text{as $x_N \to +\infty$} \\
&u(x',x_N) \to 0 \qquad v(x',x_N) \to 1 \qquad \text{as $x_N \to -\infty$}, \\
\end{split}
\end{equation}
uniformly with respect to $x'\in \mathbb{R}^{N-1}$. Then, 
\[
u(x) = \frac{1 + \tanh\big (\frac {x_N +\alpha }{\sqrt 2}\big)}{2}, \qquad v(x) = \frac{1 -\tanh\big (\frac {x_N + \alpha}{\sqrt 2}\big)}{2}
\]
for some $ \alpha \in \R$, i.e., $u$ and $v$ are $1$-dimensional and monotone and the solution is unique and explicit, up to translations. 
\end{theorem}

Restricting our attention to positive solutions, we can provide additional results.

\begin{theorem}\label{thm: monot 3}
Assume $N \ge 1$, and let $(u,v)$ be a solution of 
\begin{equation}\label{sistema3pm}
\begin{cases}
-\Delta u\,=u-u^3- 3 uv^2& \text{in}\quad\mathbb{R}^N  \\
-\Delta v\,=v-v^3- 3 u^2v& \text{in}\quad\mathbb{R}^N  \\
\quad u,v >0 & \text{in}\quad\mathbb{R}^N.
\end{cases}
\end{equation}
Then:
\begin{itemize}
\item [($i$)] $\partial_N u >0$ on $\R^N$ if and only if $\partial_N v <0$ on $\R^N$.
\item [($ii$)] If $N \le 3$ and $ \partial_N u >0$ on $\R^N$, then
\[
u(x) = \frac{1 + \tanh\big (\frac {a\cdot x +\alpha }{\sqrt 2}\big)}{2}, \qquad v(x) = \frac{1 -\tanh\big (\frac {a\cdot x+ \alpha}{\sqrt 2}\big)}{2}
\]
for some unit vector $a$ such that $ a_N >0$ and some $ \alpha \in \R$.

\item[($iii$)] If $N \le 8$, $\partial_N u >0$ on $\R^N$ and
\[
\lim_{x_N \to +\infty} u(x',x_N) =1, \quad \lim_{x_N \to -\infty} u(x',x_N) = 0
\]
point-wisely for every $x' \in \R^{N-1}$, then the same conclusion of point ($ii$) holds.
\end{itemize}

For $N >8$, problem \eqref{sistema3pm} posseses a solution $(u,v)$ satisfying $\partial_N u >0, \, \partial_N v <0$ on $\R^N$ and
\[
\lim_{x_N \to +\infty} u(x',x_N) =1, \quad \lim_{x_N \to -\infty} u(x',x_N) = 0
\]
\[
\lim_{x_N \to +\infty} v(x',x_N) =0, \quad \lim_{x_N \to -\infty} v(x',x_N) = 1
\]
point-wisely for every $x' \in \R^{N-1}$, and which is not 1-dimensional. 
\end{theorem}

\begin{remark}
1) Theorems \ref{sopra3} and \ref{thm: monot 3} are, respectively, positive answers to Gibbons' and De Giorgi's conjecture for system \eqref{sistema3}. The proofs rely on a complete characterization of solutions to \eqref{sistema3} and \eqref{sistema3pm} in terms of pairs of solutions to the Allen-Cahn equation, see Propositions \ref{Lam=3} and \ref{SolPos} below. It will then be evident that further conclusions could be obtained combining our method and the main results in \cite{FaVaTAMS}. We do not write down explicit statements only for the sake of brevity.

2) Item ($i$) (whence all the other conclusions follow) is false without the sign condition $u,v>0$ in $\R^N$. A counterexample is provided by the $1$-dimensional solution
\[
u(x) = \frac{\tanh\left(\frac{x_N}{\sqrt{2}}\right) + \tanh\left(\frac{x_N+\alpha}{\sqrt{2}}\right)}2 \quad \quad v(x) = \frac{\tanh\left(\frac{x_N}{\sqrt{2}}\right) - \tanh\left(\frac{x_N+\alpha}{\sqrt{2}}\right)}2,
\]
with $\alpha>0$. That this is a solution follows by the forthcoming Proposition \ref{Lam=3} (or it can be checked by direct computations), and it is immediate to observe that $u$ changes sign with $\pa_N u$ positive everywhere, while $v$ is negative with $\pa_N v$ sign changing. 
\end{remark}

If $\Lambda \neq 3$, the explicit classification above is out of reach, but it is natural to wonder whether or not the above results hold for generic $\Lambda >1$. This is left as an open problem. In this direction, it may be useful to observe that we can still obtain further information about the solutions using the comparison with $\Lambda=3$.

\begin{thm}\label{thm: la non 3}
Assume $N \ge 1$, $\Lambda>1$, and let $(u,v)$ be a solution of \eqref{system}. We have:
\begin{itemize}
\item[($i$)] If $ \Lambda > 3$, then
\[
u+v < 1 \qquad \text{in} \qquad \R^N.
\]
\item [($ii$)] If $ \Lambda < 3$, then
\[
u+v > 1 \qquad \text{in} \qquad \R^N.
\]
\end{itemize}
\end{thm}

The rest of the paper is devoted to the proof of the previous results. Before proceeding, we complete the introduction with further references related to our study.

Symmetry of solutions to elliptic systems of gradient type
\begin{equation}\label{general}
\Delta (u,v) = \nabla W(u,v) \qquad \text{in $\R^N$}
\end{equation}
with multi-well potential has attracted increasing attention in the last years. Beyond the aforementioned results regarding problem \eqref{syst solito}, we refer in particular to \cite{Di, FaGh}, where the authors proved, under some assumptions on the potential $W$, symmetry for monotone or stable solutions in dimension $N=2$ or $N = 3$, and to \cite{AlFu}, where the authors investigated rigidity properties of \emph{minimal solutions} to suitable \emph{symmetric systems}; see also \cite{De, Di, DiPi, Faz, Fu} for related results in low dimension, regarding more general operators and possibly unbounded solutions. Notice that system \eqref{system} falls within the general case \eqref{general}, with
\begin{equation}\label{pot}
W(u,v) = \frac{(u^2-1)^2}{4} + \frac{(v^2-1)^2}{4}  + \frac{\Lambda}2 u^2 v^2.
\end{equation}
%\begin{equation}\label{pot}
%W(u,v) = \frac{u^2(u^2-2)}{4} + \frac{v^2(v^2-2)}{4} + \frac{\Lambda}2 u^2 v^2.
%\end{equation}
Then, it is not difficult to check that the main results in \cite{Di, FaGh} apply to give $1$-dimensional symmetry of any bounded solution to \eqref{system} satisfying $\pa_N u>0$ and $\pa_N v<0$ in $\R^N$ with $N \le 3$. Theorem \ref{thm: main 1} here is the first symmetry result applying to \eqref{system} and holding in any dimension. 

Since in some of the quoted papers the authors could deal with quite general potentials $W$, it seems natural to wonder if we can relax our assumptions as well. The answer to this question is essentially negative: on one side, from the proof of Theorem \ref{thm: main 1} it will be evident that we could replace $W$ in \eqref{pot} with 
\[
W(u,v) = f(u) + g(v) + \frac{\Lambda}2 u^2 v^2,
\]
with $f(s)$ and $g(s)$ behaving like $\frac{(s^2-1)^2}{4}$; in particular, our method does not rely on the symmetry of the system with respect to the involution $(u,v) \mapsto (v,u)$. But on the other hand, there is no hope to work in a completely general setting, since the results in \cite{AlaBroGui} provide counterexample to Gibbons-type results such as Theorem \ref{thm: main 1} for multi-well potential systems, under particular assumptions on $W$ (clearly such assumptions are not satisfied in the setting studied here, see Remark \ref{rem: su Alama} for more details).

\textbf{Structure of the paper.} Sections \ref{sec: univ bound}-\ref{sec: sym} are devoted to the proof of Theorem \ref{thm: main 1}. As we have already anticipated, the first step consist in the derivation of the universal estimates collected in Theorem \ref{universal upper bound}, and is the object of Section \ref{sec: univ bound}. \newline As second step, we aim at applying the moving planes method to deduce that $u$ and $v$ have the desired monotonicity in $x_N$. Due to the competitive nature of the problem ($\Lambda$ is positive), several complications arise when trying to apply the moving plane argument (in particular, it is not easy to show that the moving plane method can start). Following the scheme used in \cite{FaSo}, in Section \ref{sec: monot} we first study the ``montonicity as $x_N \to +\infty$" with an ad-hoc argument, and then we carry on the moving planes to obtain the monotonicity in the whole space. Although the general strategy is similar to the one in \cite{FaSo}, most of the intermediate proofs are completely different. \newline Afterwards, we turn to the monotonicity in all the directions of the upper hemi-sphere 
\[
\S^{N-1}_+:=\left\{ \nu \in \S^{N-1}: \langle \nu, e_N \rangle >0\right\},
\]
adapting the argument firstly introduced in \cite{Fa99}. This is the object of Section \ref{sec: sym}, and gives $1$-dimensional symmetry completing the proof of Theorem \ref{thm: main 1}. 

Afterwards, we focus on the case $\Lambda \in (0,1]$, proving Theorems \ref{sopra<1} and \ref{sopra1} in Section \ref{sec: la < 1}. The former result is obtained using in a decisive way the estimate in item ($iii$) of Theorem \ref{universal upper bound}, while the latter one is essentially a consequence of the Liouville theorem for superharmonic functions in $\R^2$. 

Finally, we analyze the special case $\Lambda=3$ in Section \ref{sec: la = 3}, proving Theorems \ref{sopra3}, \ref{thm: monot 3} and \ref{thm: la non 3}. The main step in the proof is the complete characterization of solutions to \eqref{sistema3} in terms of pairs of solutions to the Allen-Cahn equation, whence we will easily deduce our thesis.

\section{Universal bounds}\label{sec: univ bound}

In this section we prove Theorem \ref{universal upper bound}. In the next proof and in the rest of the paper, we denote by ${\mathds 1}_A$ the characteristic function of the set $A$.

\begin{proof} Let us start with item ($i$). Since $(u,v) \in L^3_{loc}(\R^N) \times  L^3_{loc}(\R^N)$ then both $u-1$ and $\Delta (u-1)$ belong to $L^1_{loc}(\R^N)$. Therefore we can apply Kato's inequality \cite{Brezis,Kato} to $u-1$ to obtain
$$
\Delta (u-1)^+ \ge \Delta u \, {\mathds 1}_{\{ u-1>0\}} = (u(u^2 -1)+ \Lambda uv^2) \, {\mathds 1}_{\{ u-1>0\}} \ge  [(u-1)^+]^2
$$
in the sense of distribution on $\R^N$. The latter and Lemma 2 of \cite{Brezis} imply that $ (u-1)^+ \le 0$ on $\R^N$, i.e., $ u \le 1 $ a.e. on $\R^N$. The same argument applied to $v-1$ also yields that $ v\le 1 $ a.e. on $\R^N$. Finally, we observe that also $(-u,-v)$ is a solution of \eqref{systemUE} and so, the above argument implies that $ -u \le 1 $ and $ -v \le 1$ a.e. on $\R^N$. By summarizing, we proved that
\[
\vert u \vert \le 1, \quad \vert v \vert \le 1 \qquad \text{a.e. in $\R^N$},
\]
and thus $u$ and $v$ are smooth by standard elliptic  estimates.

To prove item ($ii$) we use the smoothness of $u$ and $v$ and, once again, Kato's inequality applied to the function $ u^2 + v^2 -1$:
\begin{align*}
 \Delta ( u^2 + v^2 -1)^+ & \ge (2u \Delta u + 2v \Delta v) {\mathds 1}_{\{ u^2+v^2-1>0\}}  \\
 & = 2[(u^4 +v^4 + 2 \Lambda u^2v^2) - (u^2+v^2)] {\mathds 1}_{\{ u^2+v^2-1>0\}} \\
& \ge 2[(u^2 + v^2)^2 -(u^2 +v^2)]{\mathds 1}_{\{ u^2+v^2-1>0\}} \ge [ ( u^2 + v^2 -1)^+]^2
\end{align*}
since $\Lambda \ge 1$. Therefore, $ ( u^2 + v^2 -1)^+ \le 0$, which gives the desired conclusion.

To prove item ($iii$) we observe, as before, that
$$ \Delta \left( u^2 + v^2 -\frac{2}{\Lambda +1} \right)^+ \ge  2[(u^4 +v^4 + 2 \Lambda u^2v^2) - (u^2+v^2)] {\mathds 1}_{\left\{ u^2+v^2-\left(\frac{2}{\Lambda +1} \right)>0\right\}} $$
%$$= 2[(u^2 + v^2)^2 +2(\Lambda -1)u^2v^2 -(u^2 +v^2)]{\mathds 1}_{\{ %u^2+v^2-(\frac{2}%{\Lambda +1} )>0\}} $$
and that, for $ \Lambda \in (0,1) $, the following algebraic inequality holds true :
\[
(u^4 +v^4 + 2 \Lambda u^2v^2) \ge \frac{\Lambda +1}{2} (u^2 + v^2)^2.
\]
%Since $ \vert u \vert, \, \vert v \vert \le 1$ by item i),  we have that $ 2u^2v^2 \le 2 \vert u %\vert \vert v \vert \le u^2 +v^2$ and so $ 2(\Lambda -1) u^2v^2 \ge (\Lambda -1) u^2 %+v^2$.
Hence
\begin{align*}
\Delta \left( u^2 + v^2 -\frac{2}{\Lambda +1} \right)^+ & \ge 2 \left[\frac{\Lambda +1}{2} (u^2 + v^2)^2 - (u^2+v^2) \right] {\mathds 1}_{\left\{ u^2+v^2-\left(\frac{2}{\Lambda +1} \right)>0\right\}} \\
& \ge (\Lambda +1) (u^2 + v^2) \Big [ (u^2 + v^2) -\frac{2}{\Lambda +1}\Big ]{\mathds 1}_{\left\{ u^2+v^2-\left(\frac{2}{\Lambda +1} \right)>0\right\}} \\
& \ge \left[\left(u^2 + v^2 -\frac{2}{\Lambda +1} \right)^+\right]^2,
\end{align*}
and we conclude as above.
\end{proof}

\section{Monotonicity with respect to $x_N$}\label{sec: monot}

The purpose of this section consists in showing that any $(u,v)$ fulfilling the assumptions of Theorem \ref{thm: main 1} has the desired monotonicity with respect to $x_N$.

\begin{prop}\label{prop: mon in x_N}
Under the assumptions of Theorem \ref{thm: main 1}, we have that
\[
\pa_N u>0 \quad \text{and} \quad \pa_N v<0 \qquad \text{in $\R^N$}.
\]
\end{prop}

The proof is based upon an application of the moving planes method, suitably adapted in order to deal with a non-cooperative system. For $\lambda \in \R$, we set
\[
u_{\lambda}(x',x_N):=u(x',2\lambda - x_N) \quad \text{and} \quad \Sigma_\la:= \{x_N > \la\}.
\]
We aim at proving that
\begin{equation}\label{moving thesis 0}
u_\la(x) \le u(x) \quad \text{and} \quad v_\la(x) \ge v(x) \quad \forall x \in \Sigma_\la, \ \forall \la \in \R,
\end{equation}
This and the strong maximum principle give the thesis of Proposition \ref{prop: mon in x_N}.

In order to prove that \eqref{moving thesis 0} is satisfied, we show that
\begin{equation}\label{moving thesis 1}
\Theta:=\left\{ \la \in \R: \text{$u_\theta \le u$ and $v_\theta \ge v$ in $\Sigma_\theta$ for every $\theta \ge \la$}\right\}= \R.
\end{equation}
This can be done in two steps: at first we prove that $\Theta \neq \emptyset$, so that by definition it is an unbounded interval of type $(\tilde \lambda, +\infty)$ (or eventually $[\tilde \lambda, +\infty)$). In a second time, we show that necessarily $\tilde \lambda =-\infty$.

For the first step, it is useful to recall the following statement:

\begin{lem}[{\cite[Lemma 2.1]{FMS}}]\label{bLe:L(R)}
Let $\theta >0$ and $\gamma>0$ such that $\theta < 2^{-\gamma}$.
Moreover let $R_0>0$, $C>0$ and
$$\mathcal{L}:(R_0, + \infty) \rightarrow \mathbb{R}$$ a non-negative and non-decreasing function such that
%\begin{equation}\label{beq:L}
\[ 
\begin{cases} \mathcal{L}(R)\leq \theta \mathcal{L}(2R)+g(R) & \forall R>R_0,\\ \mathcal{L}(R)\leq CR^{\gamma} & \forall R >R_0, \end{cases}
\]
%\end{equation} 
where $g:(R_0, +\infty)\rightarrow \mathbb{R}^+$ is such that $$\lim_{R\rightarrow +\infty}g(R)=0 .$$ Then $$\mathcal{L}(R)=0.$$
\end{lem}

The lemma will be used in the next proof.

\begin{lem}\label{monot at infinity}
There exists $\bar \lambda \in \R$ sufficiently large such that
\[
u\geq u_\lambda \quad\text{and}\quad v\leq v_\lambda \qquad \text{in}\,\,\Sigma_\lambda
\]
for any $\lambda\geq\bar\lambda$.
%Furthermore
%\[
%u_{x_N}> 0 \,\,\text{in}\,\,\Sigma_{\bar\lambda}\quad\text{and}\quad v_{x_N}< 0 \,\,\text{in}\,\,\Sigma_{\bar\lambda}\,.
%\]
\end{lem}
\begin{proof}
The pair $(u_\lambda,v_\lambda)$ solves
\begin{equation}\tag{$P_\lambda$}\label{syst lambda}
\begin{cases}
-\Delta u_\lambda\,=u_\lambda-u_\lambda^3-\Lambda u_\lambda v_\lambda^2& \text{in}\quad\mathbb{R}^N  \\
-\Delta v_\lambda\,=v_\lambda-v_\lambda^3-\Lambda u_\lambda^2v_\lambda& \text{in}\quad\mathbb{R}^N  \\
0<u_\lambda,v_\lambda<1 &  \text{in}\quad\mathbb{R}^N \,.
\end{cases}
\end{equation}

Let $\varphi_R$ be a standard $C^1$ cut off function such that $\varphi_R=1$ in $B_R$, $\varphi_R=0$ outside $B_{2R}$, with $|\nabla\varphi_R|\leq 2/R$. We consider the test functions
\begin{equation}\nonumber
(u_\lambda-u)^+\varphi_R^2 \, \mathds{1}_{\{x_N\geq \lambda\}},\qquad(v-v_\lambda)^+\varphi_R^2\, \mathds{1}_{\{x_N\geq \lambda\}}\,.
\end{equation}
Let us set $g(t)=t-t^3$. By \eqref{system} and \eqref{syst lambda}, using the test functions above and subtracting, we deduce that
\begin{equation}\nonumber
\begin{split}
\int_{\Sigma_\lambda}\,|\nabla (u_\lambda-u)^+|^2\varphi_R^2\,&=\,-2\int_{\Sigma_\lambda} \varphi_R (u_\lambda-u)^+ \, \nabla (u_\lambda-u)^+ \cdot \nabla\varphi_R \\
&+\int_{\Sigma_\lambda}\big(g(u_\lambda)-g(u)\big)(u_\lambda-u)^+\varphi_R^2\\
&-\int_{\Sigma_\lambda}\big(\Lambda u_\lambda v_\lambda^2-\Lambda uv^2\big)(u_\lambda-u)^+\varphi_R^2\,.
\end{split}
\end{equation}
Now we make the first assumption on $\bar\lambda$ and we suppose that
$\lambda\geq\bar\lambda$, with $\bar\lambda$ large enough in order that
\begin{equation}\nonumber
u \geq\sqrt {2/3}\quad\text{in}\,\,\Sigma_\lambda.
\end{equation}
In this way, for any $x$ such that $u_\lambda(x) \ge u(x)$ we have
\[
g(u_\lambda(x))-g(u(x)) = g'(\xi(x)) (u_\lambda(x)-u(x)) \le - (u_\lambda(x)-u(x))
\]
for some $\xi(x) \in [u(x), u_\lambda(x)]$. Furthermore,
\[
-(\Lambda u_\lambda v_\lambda^2-\Lambda uv^2\big)(u_\lambda-u)^+=-\Lambda v_\lambda^2((u_\lambda-u)^+)^2+\Lambda u (v^2-v_\lambda^2)(u_\lambda-u)^+,
\]
so that, setting
\[
\mathcal C_R\,:=\, \Sigma_\lambda\cap B_R,
\]
we deduce that for any $\theta \in (0,1)$
\begin{equation}\nonumber
\begin{split}
\int_{\mathcal C_R}\,|\nabla (u_\lambda-u)^+|^2\,&\leq \theta \int_{\mathcal C_{2R}}\,|\nabla (u_\lambda-u)^+|^2\\
&+\frac{1}{\theta R^2} \int_{\mathcal C_{2R}}\, ((u_\lambda-u)^+)^2\varphi_R^2-\int_{\mathcal C_{2R}}\, ((u_\lambda-u)^+)^2\varphi_R^2\\
&+\int_{\mathcal C_{2R}}\Lambda u(v+v_\lambda)(v-v_\lambda)^+(u_\lambda-u)^+\varphi_R^2.
\end{split}
\end{equation}
To estimate the last term on the right hand side, we observe that $(v+v_\lambda)\leq 2v$ in the support of $(v-v_\lambda)^+$; then, recalling also that $0<u,u_\lambda<1$ in $\R^N$, we obtain
\begin{equation}\nonumber
\begin{split}
\int_{\mathcal C_{2R}}\Lambda u(v+v_\lambda)(v-v_\lambda)^+(u_\lambda-u)^+\varphi_R^2 & \leq \Lambda \|v\|_{L^\infty(\Sigma_\lambda)} \int_{\mathcal C_{2R}}\,\big( (u_\lambda-u)^+\big)^2\varphi_R^2\\
& +\Lambda \|v\|_{L^\infty(\Sigma_\lambda)}\int_{\mathcal C_{2R}}\,\big( (v-v_\lambda)^+\big)^2\varphi_R^2.
\end{split}
\end{equation}
As a consequence
\begin{equation}\label{C2}
\begin{split}
\int_{\mathcal C_R}\,|\nabla (u_\lambda-u)^+|^2\,&\leq \theta \int_{\mathcal C_{2R}}\,|\nabla (u_\lambda-u)^+|^2\\
&+\frac{1}{\theta R^2} \int_{\mathcal C_{2R}}\, ((u_\lambda-u)^+)^2\varphi_R^2-\int_{\mathcal C_{2R}}\, ((u_\lambda-u)^+)^2\varphi_R^2\\
&+ \Lambda \|v\|_{L^\infty(\Sigma_\lambda)} \int_{\mathcal C_{2R}}\,\big( (u_\lambda-u)^+\big)^2\varphi_R^2\\
&+\Lambda \|v\|_{L^\infty(\Sigma_\lambda)}\int_{\mathcal C_{2R}}\,\big( (v-v_\lambda)^+\big)^2\varphi_R^2.\\
\end{split}
\end{equation}
We proceed further with similar estimates on $(v-v_\lambda)^+$. As above, by \eqref{system} and \eqref{syst lambda}
\begin{equation}\nonumber
\begin{split}
\int_{\Sigma_\lambda}\,|\nabla (v-v_\lambda)^+|^2\varphi_R^2\,&=\,-2\int_{\Sigma_\lambda} \varphi_R (v-v_\lambda)^+ \, \nabla (v-v_\lambda)^+\cdot \nabla\varphi_R\\
&+\int_{\Sigma_\lambda}\big(g(v)-g(v_\lambda)\big)(v-v_\lambda)^+\varphi_R^2\\
&+\int_{\Sigma_\lambda}\big(\Lambda u_\lambda^2 v_\lambda-\Lambda u^2v\big)(v-v_\lambda)^+\varphi_R^2\\
&=-2\int_{\Sigma_\lambda}\varphi_R (v-v_\lambda)^+ \, \nabla (v-v_\lambda)^+\cdot \nabla\varphi_R\\
&+\int_{\Sigma_\lambda}\big(g(v)-g(v_\lambda)\big)(v-v_\lambda)^+\varphi_R^2-\Lambda u^2((v-v_\lambda)^+)^2\varphi_R^2\\
&+\Lambda\int_{\Sigma_\lambda}
v_\lambda(u_\lambda^2-u^2)
(v-v_\lambda)^+\varphi_R^2\,.
\end{split}
\end{equation}
Let now $\varepsilon>0$ be such that $1+\varepsilon< \Lambda$ (recall that $\Lambda>1$). By assumption \eqref{uniform limit}, if necessary enlarging $\bar\lambda$, we can suppose that for any $\lambda> \bar \lambda$
\begin{equation}\label{cond on lambda 2}
u^2>\frac{1+\varepsilon}{\Lambda} \quad \text{in}\,\,\Sigma_\lambda, \quad \text{and} \quad 2 \Lambda \|v\|_{L^\infty(\Sigma_\lambda)} < \min\{1,\varepsilon\}.
\end{equation}
Using the first of these conditions, and the fact that $g'(t) \le 1$ for every $t \in \R$, and that $v_\lambda\leq v$ in the support of $(v-v_\lambda)^+$, we infer that
\begin{equation}\label{C1}
\begin{split}
\int_{\mathcal C_R}\,|\nabla (v-v_\lambda)^+|^2\,&\leq \theta \int_{\mathcal C_{2R}}\,|\nabla (v-v_\lambda)^+|^2\\
&+\frac{1}{\theta R^2} \int_{\mathcal C_{2R}}\, ((v-v_\lambda)^+)^2\varphi_R^2-\varepsilon\int_{\mathcal C_{2R}}\, ((v-v_\lambda)^+)^2\varphi_R^2\\
&+2\Lambda\int_{\mathcal C_{2R}}
v(u_\lambda-u)^+
(v-v_\lambda)^+\varphi_R^2 \\
%\,.\\
%\end{split}
%\end{equation}
%Then, by Young's inequality
%\begin{equation}\label{C1}
%\begin{split}
%\int_{\mathcal C_R}\,|\nabla (v-v_\lambda)^+|^2\,
%
&\leq \theta \int_{\mathcal C_{2R}}\,|\nabla (v-v_\lambda)^+|^2\\
&+\frac{1}{\theta R^2} \int_{\mathcal C_{2R}}\, ((v-v_\lambda)^+)^2\varphi_R^2-\varepsilon\int_{\mathcal C_{2R}}\, ((v-v_\lambda)^+)^2\varphi_R^2\\
&+ \Lambda \|v\|_{L^\infty(\Sigma_\lambda)}\int_{\mathcal C_{2R}}\,\big( (u_\lambda-u)^+\big)^2\varphi_R^2\\
&+\Lambda \|v\|_{L^\infty(\Sigma_\lambda)}\int_{\mathcal C_{2R}}\,\big( (v-v_\lambda)^+\big)^2\varphi_R^2.\\
\end{split}
\end{equation}

Now we set
\[
\mathcal L_\lambda(R)\,:=\, \int_{\mathcal C_R}\,|\nabla (u_\lambda-u)^+|^2\,+\,\int_{\mathcal C_R}\,|\nabla (v-v_\lambda)^+|^2,
\]
observing that $\mathcal L_\lambda(R)\leq C R^N$ since $|\nabla u|\in L^\infty(\mathbb{R}^N)$ and $|\nabla v|\in L^\infty(\mathbb{R}^N)$.
Having in mind Lemma \ref{bLe:L(R)}, we fix $\theta:= 2^{-(N+1)}$.
%\[
%\theta\,:=\,\frac{1}{2^{N+1}}\,.
%\]
Adding \eqref{C2} and \eqref{C1}, we deduce that for any $\varepsilon>0$ fixed as above, $\lambda>\bar \lambda$ fixed and $R>1$, it results that
%\begin{equation}\label{C23}
\[
\begin{split}
\mathcal L_\lambda (R)\leq \theta\mathcal L (2R)
%&+\frac{1}{\theta R^2} \int_{\mathcal C_{2R}}\, ((u_\lambda-u)^+)^2\varphi_R^2+\frac{1}{\theta R^2} \int_{\mathcal C_{2R}}\, ((v-v_\lambda)^+)^2\varphi_R^2\\
&
+ \left(2\Lambda \|v\|_{L^\infty(\Sigma_\lambda)}+\frac{1}{\theta R^2}-1\right) \int_{\mathcal C_{2R}}\,\big( (u_\lambda-u)^+\big)^2\varphi_R^2\\
&+\left(2\Lambda \|v\|_{L^\infty(\Sigma_\lambda)}+\frac{1}{\theta R^2}-\varepsilon\right)\int_{\mathcal C_{2R}}\,\big( (v-v_\lambda)^+\big)^2\varphi_R^2.
\end{split}
\]
%\end{equation}
Recalling \eqref{cond on lambda 2}, this implies that for sufficiently large $R$ we have
\[
\mathcal L_\lambda (R)\leq \theta \mathcal L_\lambda (2R)\,.
\]
Therefore we are in position to apply Lemma \ref{bLe:L(R)}, and we conclude that
\[
\mathcal L_\lambda (R)=0.
\]
This holds for any $\lambda>\bar \lambda$, and hence, recalling that $u=u_\lambda$ and $v=v_\lambda$ on $\pa \Sigma_\lambda$, the proof is complete.
%that
%\[
%u\geq u_\lambda \,\,\text{in}\,\,\Sigma_\lambda\quad\text{and}\quad v\leq v_\lambda \,\,\text{in}\,\,\Sigma_\lambda\,,
%\]
%for $\lambda\geq\bar\lambda$.
%This immediately implies that
%\[
%u_{x_N}\geq 0 \,\,\text{in}\,\,\Sigma_{\bar\lambda}\quad\text{and}\quad v_{x_N}\leq 0 \,\,\text{in}\,\,\Sigma_{\bar\lambda}\,.
%\]
%Then we have
%\begin{equation}\nonumber
%\begin{split}
%-\Delta u_{x_N}\,&=g'(u)u_{x_N} -\Lambda u_{x_N}v^2-2\Lambda uvv_{x_N}\\
%&\geq g'(u)u_{x_N} -\Lambda u_{x_N}v^2\geq - K  u_{x_N}\,.
%\end{split}
%\end{equation}
%for some $K>0$.
%By the strong maximum principle and $(h_\infty)$, it follows that
%\[
%u_{x_N}> 0 \,\,\text{in}\,\,\Sigma_{\bar\lambda}\,.
%\]
%In the same way
%\begin{equation}\nonumber
%\begin{split}
%-\Delta v_{x_N}\,&=g'(v)v_{x_N} -\Lambda u^2v_{x_N}-2\Lambda uvu_{x_N}\\
%&\leq g'(v)v_{x_N} -\Lambda u^2v_{x_N}\\
%&\leq - K  v_{x_N}
%\end{split}
%\end{equation}
%and again
%by the strong maximum principle and $(h_\infty)$, it follows that
%\[
%v_{x_N}< 0 \,\,\text{in}\,\,\Sigma_{\bar\lambda}\,.
%\]
\end{proof}

As already observed, Lemma \ref{monot at infinity} implies that the quantity $\tilde \lambda:= \inf \Theta$ (with $\Theta$ defined in \eqref{moving thesis 1}) is either a real number, or $-\infty$. We can actually rule out the former possibility, thus completing the proof of Proposition \ref{prop: mon in x_N}.

\begin{lem}
It results that $\tilde \lambda =-\infty$.
\end{lem}
\begin{proof}
The proof of this fact is similar to the one of step 2 of Proposition 5.1 in \cite{FaSo}. We report the details for the sake of completeness. Assume by contradiction that $\tilde \lambda> -\infty$ is a real number. In this case by continuity $\Theta = [\tilde \lambda,+\infty)$, and by definition of $\inf$ there exist sequences $(\lambda_i) \subset (-\infty,\tilde \lambda)$ and $(x^i) \subset \Sigma_{\la_i}$ such that $\la_i \to \tilde \lambda$ as $i \to \infty$, and at least one between
\begin{subequations}
\begin{align}
& u_{\la_i}(x^i) > u(x^i)  \label{eq8}\\
& v_{\la_i}(x^i)<v(x^i) \label{eq13}
\end{align}
\end{subequations}
holds true for every $i$.

Assume that \eqref{eq8} holds true. We claim that the sequence $(x_N^i) \subset \R$ is bounded. If not, as $x_N^i>\la_i$ and $\la_i$ is bounded, up to a subsequence $x_N^i \to +\infty$ as $i \to \infty$. It follows that $2 \la_i -x_N^i \to -\infty$, and in light of assumption \eqref{uniform limit} we obtain
\[
\lim_{i \to \infty} u_{\la_i}(x^i) = \lim_{i \to \infty} u((x^i)',2\la_i -x_N^i)  = 0 \quad \text{and} \quad \lim_{i \to \infty} u(x^i) = 1,
\]
in contradiction with \eqref{eq8} for $i$ sufficiently large. Hence the claim is proved and, up to a subsequence, $x_N^i \to x_N^\infty$ as $i \to \infty$.

Let us set
\[
u^i(x):= u((x^i)'+x',x_N) \quad \text{and} \quad v^i(x):= v((x^i)'+x',x_N).
\]
Since $(u,v)$ is bounded, by standard gradient estimates $|\nabla u|, |\nabla v| \in L^\infty(\R^N)$. Thus $\{(u^i,v^i)\}$ is uniformly bounded and equi-Lipschitz-continuous, and by elliptic estimates up to a subsequence $(u^i,v^i)$ converges via a diagonal process in $\mathcal{C}^2_{loc}(\R^N)$ to a limit $(u^\infty, v^\infty)$, still solution of \eqref{system} in $\R^N$.

We wish to show that $x_N^\infty=\tilde \lambda$. From the absurd assumption, equation \eqref{eq8}, we obtain
\begin{equation}\label{eq9}
\begin{split}
u_{\tilde \lambda}^{\infty}(0',x_N^\infty) &= u^\infty(0',2{\tilde \lambda} - x_N^\infty) = \lim_{i \to \infty} u( (x^i)', 2\la_i-x_N^i) \\
&= \lim_{i \to \infty} u_{\la_i}(x^i) \ge \lim_{i \to \infty} u(x^i) = u^\infty (0',x_N^{\infty}).
\end{split}
\end{equation}
On the other hand, we observe that $((x^i)'+x',x_N) \in \Sigma_{\tilde \lambda}$ whenever $(x',x_N) \in \Sigma_{\tilde \lambda}$, and by definition $u_{\tilde \lambda} \le u$ in $\Sigma_{\tilde \lambda}$. Consequently, by the convergence of $u^i$ to $u^\infty$ we deduce that
\begin{align*}
u^{\infty}_{\tilde \lambda}(x',x_N) &= \lim_{i \to \infty} u^i(x',2{\tilde \lambda} - x_N) = \lim_{i \to \infty} u((x^i)'+x',2 {\tilde \lambda}- x_N) \\
& \le \lim_{i \to \infty} u((x^i)'+x',x_N) = \lim_{i \to \infty} u^i(x',x_N) = u^\infty(x',x_N)
\end{align*}
for every $(x',x_N) \in \Sigma_{\tilde \lambda}$. Analogously, as $v_{\tilde \lambda} \ge v$ in $\Sigma_{\tilde \lambda}$, we have $v^\infty_{\tilde \lambda} \ge v^\infty$ in $\Sigma_{\tilde \lambda}$.\\
Now,
\begin{equation}\label{eq12}
\begin{cases}
-\Delta (u^\infty - u^\infty_{\tilde \lambda}) + c(x) (u^\infty- u^\infty_{\tilde \lambda})= \Lambda ((v^\infty_{\tilde \lambda})^2 - (v^\infty)^2) u^\infty_{\tilde \lambda} \ge 0 & \text{in $\Sigma_{\tilde \lambda}$} \\
u^\infty-u^\infty_{\tilde \lambda} \ge 0 & \text{in $\Sigma_{\tilde \lambda}$} \\
u^\infty -u^\infty_{\tilde \lambda} = 0 & \text{on $\pa \Sigma_{\tilde \lambda}$},
\end{cases}
\end{equation}
with $c \in L^\infty(\Sigma_{\tilde \lambda})$ defined by
\[
c(x):= \Lambda v^2(x)- c_0(x), \qquad c_0(x):= \begin{cases} \frac{g(u_{\tilde \lambda}(x))- g(u(x))}{u_{\tilde \lambda}(x)-u(x)} & \text{if $u_{\tilde \lambda}(x) \neq u(x)$} \\
g'(u(x)) & \text{if $u_{\tilde \lambda}(x) = u(x)$} \end{cases}
\]
(recall that $g(t) = t-t^3$).
%Furthermore, $u^\infty-u_{\tilde \lambda}^\infty$ is not identically $0$: indeed by assumption \eqref{uniform limit}
%\[
%\lim_{x_N \to +\infty}   u^\infty(x',x_N) - u^\infty_{\tilde \lambda}(x',x_N) = +\infty.
%\]
Hence, the strong maximum principle together with assumption \eqref{uniform limit} implies that necessarily $u^\infty- u^\infty_{\tilde \lambda} > 0$ in $\Sigma_{\tilde \lambda}$, and a comparison with \eqref{eq9} reveals that $x_N^\infty={\tilde \lambda}$, as desired.

At this point we are ready to reach a contradiction. On one side, by the absurd assumption \eqref{eq8}
\[
0< u_{\la_i}(x^i) - u(x^i) = u^i(0', 2\la_i -x^i_N)-u^i(0',x_N) = 2\pa_N u^i (0', \xi^i)(\la_i-x^i_N) \qquad \forall i,
\]
for some $\xi^i \in (2\lambda_i -x_N^i,x_N^i)$. As $\la_i < x_N^i$ for every $i$ this implies $\pa_N u^i(x', \xi^i_N)< 0$ for every $i$, and passing to the limit we infer that
\begin{equation}\label{eq11}
\pa_N u^\infty(0', {\tilde \lambda}) \le 0,
\end{equation}
where we used the fact that $\la_i \le \xi^i \le x_N^i$ with $\la_i,x_N^i \to {\tilde \lambda}$.

On the other side, thanks to  \eqref{eq12} and the fact that $u^\infty-u^\infty_{\tilde \lambda}>0$ in $\Sigma_{\tilde \lambda}$, the Hopf Lemma implies that
\[
-2\pa_N u^\infty(0',{\tilde \lambda}) =\pa_{-e_N} (u^\infty(0',{\tilde \lambda}) -u^\infty_{\tilde \lambda}(0',{\tilde \lambda}))<0,
\]
in contradiction with \eqref{eq11}.

The above argument establishes that \eqref{eq8} cannot occur. With minor changes, we can show that also \eqref{eq13} cannot be verified, and in conclusion $\tilde \lambda$ cannot be finite.
\end{proof}

\begin{proof}[Proof of Proposition \ref{prop: mon in x_N}]
By \eqref{moving thesis 1}, we directly deduce that $\pa_N u \ge 0$ and $\pa_N v \le 0$ in $\R^N$. Since
\[
\begin{cases}
-\Delta (\pa_N u) + (3 u^2  + \Lambda v^2 -1) \pa_N u = -2 \Lambda uv \pa_N v \ge 0 & \text{in $\R^N$} \\
-\Delta (\pa_N v) + (3 v^2  + \Lambda u^2 -1) \pa_N v = -2 \Lambda uv \pa_N u \le 0 & \text{in $\R^N$},
\end{cases}
\]
the strict inequality follows by the strong maximum principle and \eqref{uniform limit}.
\end{proof}

\begin{remark}\label{rem: su Alama}
In this section we established the monotonicity of any solution to \eqref{system} satisfying assumption \eqref{uniform limit}. In particular, this result holds in dimension $N=1$, and combined with Theorem 1.3 in \cite{AftSou} implies the uniqueness of the positive heteroclinic connection between $(0,1)$ and $(1,0)$. This permits to verify that, in the setting considered here, the assumptions in \cite{AlaBroGui} are not satisfied. Indeed, one of the crucial assumption in \cite{AlaBroGui} is the non-uniqueness (modulo translations) of the 1D minimal heteroclinic connection between two global minima of the potential, in this case $(0,1)$ and $(1,0)$. But any such connection is positive by minimality (as shown in \cite[page 582]{AlaBro1}), and hence unique by the above discussion. In addition, we can also observe that to apply \cite[Theorem 1.1]{AlaBroGui} one need the odd-symmetry of one component of the connection, which cannot be verified by positive solutions.
\end{remark}

%\begin{remark}
%We mention that the uniqueness of the 1D heteroclinic connection was conjectured in \cite[Section 5]{AlaBro1}, and proved in \cite{AftSou} within the class of solutions with a monotone component. Proposition \ref{prop: mon in x_N} permits to remove the monotonicity assumption, thus closing the conjecture in complete generality, as stated in Corollary \ref{corol uniqueness} above.
%\end{remark}

\section{$1$-dimensional symmetry}\label{sec: sym}

In this section we pass from the monotonicity in $x_N$ to the monotonicity in all the directions of the upper hemi-sphere $\mathbb{S}^{N-1}_+:=\left\{ \nu \in \mathbb{S}^{N-1}: \langle \nu, e_N \rangle >0 \right\}$. We follow the idea introduced in \cite{Fa99} (see also \cite{FaVaInd}), which was already successively adapted in \cite{FaSo} to deal with competitive systems. In any case, with respect to \cite{FaSo}, the adaptation here presents several differences.

\begin{prop}\label{prop: symmetry}
For every $\nu \in \mathbb{S}^{N-1}_+$, we have
\[
\pa_\nu u >0 \quad \text{and} \quad \pa_\nu v<0 \qquad \text{in $\R^N$}.
\]
\end{prop}

We divide the proof in several steps.

\begin{lem}\label{lem: step 1 2}
Let $\sigma >0$ be arbitrarily chosen. There exists an open neighborhood $\mathcal{O}_{e_N}$ of $e_N$ in $\mathbb{S}^{N-1}$ such that
%\begin{equation}\label{eq16}
\[
\frac{\pa u}{\pa \nu}  (x) > 0 \quad \text{and} \quad \frac{\pa v}{\pa \nu}(x) <0 \quad \forall x \in \overline{S_\s}, \ \forall \nu \in \mathcal{O}_{e_N},
\]
%\end{equation}
where $S_\sigma:= \R^{N-1} \times (-\sigma,\sigma)$.
\end{lem}

\begin{proof}
Let $\sigma>0$ be arbitrarily chosen. At first, we claim that there exists $\eps=\eps(\sigma)>0$ such that
\begin{equation}\label{th step 1}
\pa_N u(x) \ge \eps \quad \text{and} \quad \pa_N v(x) \le -\eps \quad \forall x \in \overline{S}_\sigma.
\end{equation}
By contradiction, assume that there exists $(x^i) \subset S_\sigma$ such that at least one between
\begin{subequations}
\begin{align}
& \lim_{i \to +\infty} \pa_N u(x^i)= 0 \label{eq14} \\
& \lim_{i \to +\infty} \pa_N v(x^i)= 0 \label{eq15}
\end{align}
\end{subequations}
holds true for every $i$. Assume e.g. that \eqref{eq14} holds. We define
\[
u^i(x):= u(x+x^i) \quad \text{and} \quad v^i(x):= v(x+x^i).
\]
The sequence $\{(u^i,v^i)\}$ is uniformly bounded in $W^{1,\infty}(\R^N)$, and hence by elliptic regularity $(u^i,v^i) \to (u^\infty,v^\infty)$ in $\mathcal{C}^2_{loc}(\R^N)$ up to a subsequence and via a diagonal process, where $(u^\infty, v^\infty)$ is still a solution to \eqref{system}. By the convergence, we have
\[
\pa_N u^\infty \ge 0   \quad \text{and} \quad \pa v^\infty_N \le 0 \quad \text{in $\R^N$},
\]
and $\pa_N u^\infty(0)=0$. Furthermore,
\[
-\Delta \left(\pa_N u^\infty \right) +(3 (u^\infty)^2  + \Lambda (v^\infty)^2 -1) \left(\pa_N u^\infty \right) = -2 \Lambda u^\infty v^\infty \left(\pa_N v^\infty \right) \ge 0 \quad \text{in $\R^N$}.
\]
The strong maximum principle implies that either $\pa_N u^\infty>0$ or $\pa_N u^\infty \equiv 0$. The former one is in contradiction with the fact that $\pa_N u^\infty(0)=0$, the latter one is in contradiction with assumption \eqref{uniform limit}, which is also satisfied by the limit profile $(u^\infty,v^\infty)$ since $(x_N^i)$ is bounded. Thus, \eqref{eq14} cannot occur. A similar argument shows that also \eqref{eq15} does not hold, and completes the proof of claim \eqref{th step 1}.

Now we claim that
\begin{equation}\label{th step 2}
\text{The map $\nu \mapsto (\pa_\nu u, \pa_\nu v)$ is in $\mathcal{C}^{0,1} \left(\mathbb{S}^{N-1}, \left(\mathcal{C}^0(\R^N)\right)^2\right)$}.
\end{equation}
This is a simple consequence of the globlal Lipschitz continuity of $(u,v)$, which implies that
\[
\left|\frac{\pa u}{\pa \nu_1}(x) -\frac{\pa u}{\pa \nu_2}(x) \right| + \left|\frac{\pa v}{\pa \nu_1}(x) -\frac{\pa v}{\pa \nu_2}(x) \right| \le 2 C |\nu_1-\nu_2|
\]
for every $x \in \R^N$.

Combining \eqref{th step 1} and \eqref{th step 2}, the thesis follows.
\end{proof}

\begin{lem}\label{lem: step 3}
$u$ is strictly increasing and $v$ is strictly decreasing with respect to all the unit vectors of an open neighborhood of $e_N$ in $\mathbb{S}^{N-1}$.
\end{lem}

\begin{proof}
First, we write down the equations satisfied by the directional derivatives $\pa_\nu u = u_\nu$ and $\pa_\nu v= v_\nu$:
\begin{equation}\label{linearized}
\begin{cases}
-\Delta u_\nu = c_1(x) u_\nu -2\Lambda uv v_\nu \\
-\Delta v_\nu = c_2(x) v_\nu -2\Lambda uv u_\nu
\end{cases} \text{in $\R^N$},
\end{equation}
where
\[
c_1(x):= 1-\Lambda v^2(x) -3 u^2(x), \quad \text{and} \quad c_2(x):= 1-\Lambda u^2(x) -3 v^2(x).
\]
We choose $\s>0$ sufficiently large, in such a way that for a positive small $\eps$
\begin{equation}\label{cond on s}
\sup_{S_\s^c} c_1 \le -\eps, \qquad \sup_{S_\s^c} c_2 \le -\eps, \qquad \Lambda \sup_{\{x_N>\sigma\}} v < \frac{\eps}2, \qquad \Lambda \sup_{\{x_N<-\sigma\}} u < \frac{\eps}2,
\end{equation}
with $S_\s$ as in Lemma \ref{lem: step 1 2}.
The existence of such $\s$ is guaranteed by \eqref{uniform limit} and by the fact that $\Lambda>1$.

Let also $\mathcal{O}_{e_N}$ be the neighborhood of $e_N$ given by Lemma \ref{lem: step 1 2}. We test the first equation in \eqref{linearized} with $u_\nu^- \varphi_R^2$ in $\Sigma_\sigma=\{x_N > \sigma\}$, where $\varphi_R$ is chosen exactly as in Lemma \ref{monot at infinity}: writing $\mathcal{C}_{R}:= \Sigma_\sigma \cap B_{R}$, and using the fact that $u_\nu \ge 0$ on $\{x_N = \sigma\}$, we easily obtain
\begin{align*}
\int_{\mathcal{C}_R} |\nabla u_\nu^-|^2 & \le -2 \int_{\mathcal{C}_{2R}} u_\nu^- \varphi_R \nabla u_\nu^- \cdot \nabla \varphi_R  \\
& + \int_{\mathcal{C}_{2R}} c_1 (u_\nu^- \varphi_R)^2  +2 \Lambda \int_{\mathcal{C}_{2R}} \varphi_R^2 u v u_\nu^- v_\nu^+ \\
& \le \theta \int_{\mathcal{C}_{2R}} |\nabla u_\nu^-|^2 + \left( \frac{1}{\theta R^2} + \sup_{\Sigma_\s} c_1 \right) \int_{\mathcal{C}_{2R}} (u_\nu^- \varphi_R)^2 \\
& + \Lambda \sup_{\Sigma_\s} v \int_{\mathcal{C}_{2R}} \varphi_R^2 \left[(u_\nu^-)^2 + (v_\nu^+)^2\right],
\end{align*}
where $0<\theta<2^{-N}$. In a similar way, we also deduce that
\begin{align*}
\int_{\mathcal{C}_R} |\nabla v_\nu^+|^2 & \le \theta \int_{\mathcal{C}_{2R}} |\nabla v_\nu^+|^2 + \left( \frac{1}{\theta R^2} + \sup_{\Sigma_\s} c_2 \right) \int_{\mathcal{C}_{2R}} (v_\nu^+ \varphi_R)^2 \\
& + \Lambda \sup_{\Sigma_\s} v \int_{\mathcal{C}_{2R}} \varphi_R^2\left[(u_\nu^-)^2 + (v_\nu^+)^2\right].
\end{align*}
Summing up the terms in the above inequalities, we infer that
\begin{align*}
\int_{\mathcal{C}_R} |\nabla u_\nu^-|^2 + |\nabla v_\nu^+|^2 &\le \theta \int_{\mathcal{C}_{2R}} |\nabla u_\nu^-|^2 + |\nabla v_\nu^+|^2  + \left(  \frac{1}{\theta R^2} + \sup_{\Sigma_\s} c_1 + \Lambda \sup_{\Sigma_\s} v\right) \int_{\mathcal{C}_{2R}} (u_\nu^- \varphi_R)^2 \\
& + \left(  \frac{1}{\theta R^2} + \sup_{\Sigma_\s} c_2 + \Lambda \sup_{\Sigma_\s} v\right) \int_{\mathcal{C}_{2R}} (v_\nu^+ \varphi_R)^2 \\
 & \le  \theta \int_{\mathcal{C}_{2R}} |\nabla u_\nu^-|^2 + |\nabla v_\nu^+|^2
 \end{align*}
for $R$ sufficiently large, where we used estimate \eqref{cond on s} and assumption \eqref{uniform limit}. As a consequence, Lemma \ref{bLe:L(R)} is applicable, and implies that $u_\nu \ge 0$ and $v_\nu \le 0$ in $\Sigma_\s=\{x_N > \sigma\}$. Arguing exactly in the same way, we can show that the same conditions are satisfied in $\{x_N <-\sigma\}$, and finally by Lemma \ref{lem: step 1 2} we deduce that $u_\nu \ge 0$ and $v_\nu \le 0$ in $\R^N$ for every $\nu \in \mathcal{O}_{e_N}$, with both $u_\nu \not \equiv 0$ and $v_\nu \not \equiv 0$ by \eqref{uniform limit}, whence the thesis follows.
\end{proof}

\begin{proof}[Proof of Proposition \ref{prop: symmetry}]
Here we can essentially apply the same argument used in step 4 of Proposition 6.1 in \cite{FaSo}. We report the details for the sake of completeness. Let $\Omega$ be the set of the directions $\nu \in \S^{N-1}_+$ for which there exists an open neighborhood $\mathcal{O}_\nu \subset \S^{N-1}_+$ of $\nu$ such that
\[
\pa_\mu u>0 \quad \text{and} \quad \pa_\mu v<0 \quad \text{in $\R^N$}, \ \forall \mu \in \mathcal{O}_\nu.
\]
The set $\Omega$ is open by definition, and contains $e_N$ by Lemma \ref{lem: step 3}. Since $\S^{N-1}_+$ is arc-connected, if we can show that $\pa \Omega \cap \S^{N-1}_+=\emptyset$, then we conclude that $\Omega=\S^{N-1}_+$, as desired. Thus, let us suppose by contradiction that $\bar \nu \in \pa \Omega \cap \S^{N-1}_+$ (notice in particular that $\langle e_N, \bar \nu\rangle >0$). By definition, there exists $(\nu_n) \subset \Omega$ such that $\nu_n \to \bar \nu$. As
\[
\pa_{\nu_n} u >0 \quad \text{and} \quad \pa_{\nu_n} v<0 \quad \text{in $\R^N$}, \ \forall n,
\]
by continuity
 \[
\pa_{\bar \nu} u  \ge 0 \quad \text{and} \quad \pa_{\bar \nu} v\le 0 \quad \text{in $\R^N$}.
\]
By the strong maximum principle, recalling that $(u_{\bar \nu}, v_{\bar \nu})$ solves \eqref{linearized}, either $u_{\bar \nu} \equiv 0$ or $u_{\bar \nu}>0$ in $\R^N$, and analogously either $v_{\bar \nu} \equiv 0$ or $v_{\bar \nu}<0$ in $\R^N$. But the alternatives $u_{\bar \nu}  \equiv 0$ and $v_{\bar \nu} \equiv 0$ are in contradiction with assumption \eqref{uniform limit}, since $\bar \nu$ is not orthogonal to $e_N$, and hence
\begin{equation}\label{2431}
\pa_{\bar \nu} u>0\quad \text{and} \quad  \pa_{\bar \nu} v<0 \qquad \text{in $\R^N$}.
\end{equation}
Having established \eqref{2431}, it is possible to adapt the same proof of Lemmas \ref{lem: step 1 2} and \ref{lem: step 3}, with $\bar \nu$ instead of $e_N$, to deduce that $u_\nu >0$ and $v_\nu<0$ in $\R^N$ in all the direction of an open neighborhood $\mathcal{O}_{\bar \nu}$ of $\bar \nu$ in $\mathbb{S}^{N-1}_+$. Thus, we have that $\bar \nu \in \Omega \cap \pa \Omega$, in contradiction with the openness of $\Omega$.  This shows that $\pa \Omega \cap \S^{N-1}_+=\emptyset$ which, as already observed, implies $\Omega = \S^{N-1}_+$.
\end{proof}

We are ready to proceed with the:

\begin{proof}[Conclusion of the proof of Theorem \ref{thm: main 1}]
By Proposition \ref{prop: symmetry}, we immediately obtain both $\pa_{\tau} u \equiv 0$ and  $\pa_{\tau} v \equiv 0$ for every $\tau \in \S^{N-1}$ orthogonal to $e_N$.
\end{proof}

\begin{proof}[Proof of Corollary \ref{cor: uniqueness}]
By Theorem \ref{thm: main 1}, any solution to \eqref{system}-\eqref{uniform limit} is $1$-dimensional and has montone components. Therefore, the thesis follows by the uniqueness (modulo translations) of the 1D monotone heteroclinic connections proved in \cite{AftSou}.
\end{proof}

\begin{proof}[Proof of Corollary \ref{cor: sharp}]
If $(u-v) \to 1$ as $x_N \to +\infty$, then by $0 <u,v<1$ we immediately deduce that $u \to 1$ and $v \to 0$. The same discussion applies for the limit as $x_N \to -\infty$, and hence the corollary follows by Theorem \ref{thm: main 1}. 
\end{proof}

\section{The case $ \Lambda  \in (0,1]$}\label{sec: la < 1}

We discuss separately the cases $\Lambda \in (0,1)$ and $\Lambda=1$, starting from the former one.

\begin{proof}[Proof of Theorem \ref{sopra<1}]
By item ($iii$) of Theorem \ref{universal upper bound} we know that $ - v^2 \ge u^2 - \frac{2}{1+ \Lambda} $ on $\R^N$, and thus $u$ solves
\begin{equation}\label{sopraeq}
\begin{cases}
-\Delta u\, \ge u-u^3 + \Lambda u\left(u^2 - \frac{2}{1 +\Lambda}\right) = (1 - \Lambda) u \left(\frac{1}{1+ \Lambda} - u^2\right)& \text{in}\quad\mathbb{R}^N  \\
\quad u > 0 & \text{in}\quad\mathbb{R}^N
\end{cases}
\end{equation}

Given $ \gamma \in \left(0, \frac{1} {{\sqrt {1+\Lambda}}}\right)$, there is $ \delta =\delta(\Lambda, \gamma)>0$ such that
\begin{equation}\label{soprasol}
(1 - \Lambda) t \left(\frac{1}{1+ \Lambda} - t^2\right) \ge \delta t  \qquad \forall t \in [0, \gamma]
\end{equation}
here we have used, in a crucial way, that $ 1 - \Lambda >0$. Therefore, there is $ R=R(\delta)>0$ such that the principal eigenvalue $ \lambda_1=\lambda_1(B_R)$ of $ -\Delta $ in the open ball $B_R$ under zero Dirichlet boundary conditions satisfies $ \lambda_1 < \delta$.
Let $ \phi_1$ be the eigenfunction of $ -\Delta $ in $B_R$ such that $ \max_{B_R} \phi_1=1$; then, for every $ \varepsilon \in (0, \min\{ \gamma, \min_{\overline{B_R}} u\})$, 
%{\color{red} non dovrebbe essere $\varepsilon \in (0, \min\{ \gamma, \min_{\overline{B_R}} %u\})$ per poter poi applicare la \eqref{soprasol}? Comunque non cambia nulla} 
the function $w = \varepsilon \phi_1 $ satisfies
%\begin{equation}\label{sottosol}
\[
\begin{cases}
0 < w \le \varepsilon < u & \text{in}\quad B_R \\
-\Delta w \le (1 - \Lambda) w \left(\frac{1}{1+ \Lambda} - w^2\right)& \text{in}\quad  B_R \\
w = 0 & \text{in}\quad \partial B_R \\
\end{cases}
\]
%\end{equation}
since $\lambda_1 w \le \delta w \le (1 - \Lambda) w \left(\frac{1}{1+ \Lambda} - w^2\right)$ thanks to \eqref{soprasol}. Hence, in view of \eqref{sopraeq}, we can use either the sliding method (cf. e.g. \cite[Lemma 3.1]{BCN}) or the sweeping method (cf. e.g. \cite{Ser}) to infer that $ u \ge \varepsilon >0$ on the whole of space $ \R^N$.  With this information in our hands we can apply Kato's inequality to \eqref{sopraeq} to get that
\[
\Delta \left (\frac{1} {{\sqrt {1+\Lambda}}} - u \right)^+ \ge
(1 - \Lambda) \varepsilon \left [\left(\frac{1} {{\sqrt {1+\Lambda}}} - u \right)^+ \right ]^2
\]
which, in turn, provides $\Big (\frac{1} {{\sqrt {1+\Lambda}}} - u \Big)^+ \le 0$ on $ \R^N$ and so $ u^2 \ge \frac{1} {1+\Lambda}$ on $\R^N$. Since the system is symmetric in $u$ and $v$, the previous argument also gives that  $ v^2 \ge \frac{1} {1+\Lambda}$ on $\R^N$.  Comparing these informations with $ u^2 + v^2 \le \frac{2}{1+ \Lambda} $ on $\R^N$, we immediately infer that $ u^2 =v^2 = \frac{1} {1+\Lambda}$ on $\R^N$, concluding the proof.
\end{proof}

\begin{proof}[Proof of Theorem \ref{sopra1}]
By item ($ii$) of Theorem \ref{universal upper bound} we know that $u^2 + v^2 \le 1$ on $\R^N$, and hence, being $\Lambda=1$, both $u$ and $v$ are positive super-harmonic functions in $ \R^N$. Since $ N \le 2$, they must be constant.
\end{proof}

\section{The (special) case $ \Lambda =3$}\label{sec: la = 3}

In this section we first study the special system \eqref{sistema3}. Afterwards, we obtain the improved estimates of Theorem \ref{thm: la non 3} regarding the case $\Lambda \neq 3$. 

The first step in our analysis is represented by the following statement.

\begin{proposition}\label{Lam=3}
Assume $ N \ge 1$. A pair $(u,v)$ (not necessarily non-negative) is a solution of \eqref{sistema3} if and only if there exist $w_1$ and $w_2$, solutions to the Allen-Cahn equation $ - \Delta w = w-w^3$ in $ \R^N$, such that 
\begin{equation}\label{sistemaAC}
\begin{cases}
u = \frac{w_1 + w_2}{2}&\text{in}\quad\mathbb{R}^N  \\
v = \frac{w_1 - w_2}{2}& \text{in}\quad\mathbb{R}^N.
\end{cases}
\end{equation}
\end{proposition}

\begin{proof}
Since $ \Lambda =3$, by adding the two equations in \eqref{Lam=3} we immediately see that $u+v$ solves the Allen-Cahn equation
\[
\Delta (u+v) = (u+v)^3 - (u+v) \qquad \text{on} \qquad \R^N.
\]
In a similar way we see that $u-v$ solves the Allen-Cahn equation
\[
\Delta (u-v) = (u-v)^3 - (u-v) \qquad \text{on} \qquad \R^N.
\]
The desired conclusion the follows by setting $ w_1 =u+v$ and $ w_2 = u-v$. 

Conversely, it is a straightforward computation to see that the pair $(\frac{w_1 + w_2} {2}, \frac{w_1 - w_2}{2})$ is a solution of \eqref{Lam=3} if $w_1$ and $w_2$ solve the Allen-Cahn equation $ - \Delta w = w-w^3$ in $ \R^N$. 
\end{proof}

This permits us to prove Theorem \ref{sopra3}.

\begin{proof}[Proof of Theorem \ref{sopra3}]
We observe that $ w_1 = u+v \to 1 $ as $x_N \to \pm \infty$, while $ w_2 = u-v \to 1$ as $x_N \to +\infty$ and $ w_2 = u-v \to -1 $ as $x_N \to -\infty$. Theorem 1.1 of \cite{Fa} implies that $w_1 \equiv 1$ on $\R^N$ and $ w_2 = \tanh\big (\frac {x_N +\alpha }{\sqrt 2}\big)$
for some $ \alpha \in \R$. The conclusion then follows from Proposition \ref{Lam=3}.  
\end{proof}

Now we focus on positive solutions.

\begin{proposition}\label{SolPos}
Assume $ N \ge 1$. A pair $(u,v)$ is a solution of \eqref{sistema3pm} if and only if 
\[
u = \frac{1 + w_2}{2}, \qquad v = \frac{1 -w_2}{2} = 1 - u
\]
where $w_2$ is a solution of the Allen-Cahn equation $ - \Delta w = w-w^3$ in $ \R^N$, with $ w_2 \not \equiv \pm 1$. 
\end{proposition}
\begin{proof}
We observe that $ w_1 = u+v >0 $ in $ \R^N$ and so, Theorem 1.1 of \cite{Fa} implies that $w_1 \equiv 1$ on $\R^N$. The conclusion then follows from Proposition \ref{Lam=3}, by also recalling that any non-constant entire solution $w$ of the Allen-Cahn equation satisfies the universal bound $ \vert w\vert < 1$ everywhere on $ \R^N$ (\cite{FarFEQL}). 
\end{proof}

As a consequence:

\begin{proof}[Proof of Theorem \ref{thm: monot 3}]
Item ($i$) is an immediate consequence of Proposition \ref{SolPos}, where we proved that $ v = 1 -u$.

For item ($ii$), we observe that $ w_2 = u-v = 2u -1 $ in $ \R^N$, again by Proposition \ref{SolPos}. Therefore, $ \partial_N w_2 >0$ on $ \R^N$, and so the known results regarding De Giorgi's conjecture for the Allen-Cahn equation \cite{BCN2,GG,AmCa} tell us that
\[
w_2(x) = \tanh\Big (\frac {a\cdot x +\alpha }{\sqrt 2}\Big),
\]
for some unit vector $a$ such that $ a_N >0$ and some $ \alpha \in \R$. In the same way, by \cite{Sa} we can prove item ($iii$).

To prove the last claim of the Theorem we recall (see \cite{DKW}) that in dimension $ N>8$ there exists an entire solution $w_2$ of the Allen-Cahn equation satisfying $\partial_N w_2 >0$ on $\R^N$ and
\[
\lim_{x_N \to +\infty} w_2(x',x_N) =1, \quad \lim_{x_N \to -\infty} w_2(x',x_N) = -1
\]
point-wisely for every $x' \in \R^{N-1}$, and which is not 1-dimensional. Thus, an application of Proposition \ref{Lam=3} with $w_1 =1$ and the above solution $ w_2$ yields that the pair $(\frac{1 + w_2}{2},\frac{1 - w_2}{2})$ is a solution of \eqref{sistema3pm} with all the desired properties. 
\end{proof}

Finally, we give the proof of Theorem \ref{thm: la non 3}.

\begin{proof}[Proof of Theorem \ref{thm: la non 3}] We observe that, for $ \Lambda > 0, $
\begin{equation}\label{cubo}
\Delta (u+v) = (u+v)^3 - (u+v)  + (\Lambda -3)uv(u+v) \qquad \text{in} \qquad \R^N.
\end{equation}
For $ \Lambda >3$, from \eqref{cubo} and Kato's inequality we get
\[
\Delta (u+v -1)^+ \ge [(u+v)^3 - (u+v)] {\mathds 1}_{\{ u+v-1>0\}}
\ge [(u+v-1)^+]^3 \qquad \text{in} \qquad \R^N,
\]
and so $ u+v \le 1$ on $ \R^N$. Now, if $u(x_0) + v(x_0) = 1, \, x_0 \in \R^N$, then
\[
0 \ge (\Delta(u+v))(x_0) = 1 - 1 +  (\Lambda -3)u(x_0)v(x_0)(1) > 0
\]
a contradiction.

For $ \Lambda  \in (0,3)$, from \eqref{cubo} we deduce that
\[
- \Delta (u+v)  \ge (u+v) - (u+v)^3 = (u+v) [1 - (u+v)^2] \,\, \text{in} \,\, \R^N
%[1 -(u+v)]^2(1 + (u+v)) + (1 -(u+v)^2) \,\, on \,\, \R^N
\]
and, by proceeding as in the proof of Theorem \ref{sopra<1} we first see that $ u+v \ge \varepsilon$ on $\R^N$, for some $\varepsilon >0$,  and then, by Kato's inequality, we infer that
\[
\Delta (1-(u+v))^+ \ge \varepsilon [(1-(u+v))^+]^2 \qquad \text{in} \qquad \R^N
\]
which provides $u+v \ge 1$. The strict inequality then follows as before.
\end{proof}


\begin{thebibliography}{99}

\bibitem{AftSou}
\newblock A. Aftalion and C. Sourdis.
\newblock Interface leyer of a two-component Bose-Einstein condensate.
\newblock {\em Commun. Contemp. Math.}, in press. Doi: 10.1142/S0219199716500528.



\bibitem{AlaBro1}
\newblock S. Alama, L. Bronsard, A. Contreras, D.~E. Pelinovsky.
\newblock Domain walls in the coupled Gross-Pitaevskii equations.
\newblock {\em Arch. Ration. Mech. Anal.}, 215(2):579--610, 2015.

\bibitem{AlaBroGui}
\newblock S. Alama, L. Bronsard and C. Gui.
\newblock Stationary layered solutions in $\R^2$ for an Allen-Cahn system with multiple well potential.
\newblock {\em Calc. Var. Partial Differential Equations} 5(4):359--390, 1997. 



\bibitem{AlFu}
\newblock N.~D. Alikakos and G.~Fusco.
\newblock Asymptotic behavior and rigidity results for symmetric solutions of the elliptic system $\Delta u= \nabla_u W(u)$.
\newblock  {\em Ann. Sc. Norm. Super. Pisa Cl. Sci.} 15:809--836, 2016.

\bibitem{AmCa} 
\newblock L. Ambrosio and X. Cabr\'e.
\newblock Entire solutions of semilinear elliptic equations in $\R^3$ and a conjecture of De Giorgi.
\newblock {\em J. Amer. Math. Soc.}, 13(4):725--739, 2000.


\bibitem{BaBaGu} 
\newblock M.~T. Barlow, R.~F. Bass, C. Gui.
\newblock The Liouville property and a conjecture of De Giorgi.
\newblock {\em Comm. Pure Appl. Math.}, 53(8): 1007--1038, 2000.


\bibitem{BCN2} 
\newblock H. Berestycki, L. Caffarelli and L. Nirenberg.
\newblock Further qualitative properties for elliptic equations  in unbouded domains.
\newblock \emph{Ann. Scuola Norm. Sup. Pisa Cl. Sci. $(4)$} 15:69-94, 1997.


\bibitem{BCN} 
\newblock H. Berestycki, L. Caffarelli and L. Nirenberg.
\newblock Monotonicity for elliptic equations in an   unbounded Lipschitz domain.
\newblock {\em Comm. Pure Appl. Math.}, 50:1089--1112, 1997.




\bibitem{BeHaMo} 
\newblock H. Berestycki, F. Hamel and R. Monneau.
\newblock One-dimensional symmetry of bounded entire solutions of some elliptic equations.
\newblock {\em Duke Math. J.}, 103(3):375--396, 2000. 



\bibitem{BeLiWeZh}
H.~Berestycki, T.-C. Lin, J.~Wei, and C.~Zhao.
\newblock On {P}hase-{S}eparation {M}odels: {A}symptotics and {Q}ualitative
  {P}roperties.
\newblock {\em Arch. Ration. Mech. Anal.}, 208(1):163--200, 2013.

\bibitem{BeTeWaWe}
H.~Berestycki, S.~Terracini, K.~Wang, and J.~Wei.
\newblock On entire solutions of an elliptic system modeling phase separations.
\newblock {\em Adv. Math.}, 243:102--126, 2013.



\bibitem{Brezis}
\newblock H. Brezis,
\newblock \emph {Semilinear equations in $R^N$ without condition at infinity},
\newblock {\em Appl. Math. Optim.}, 12 (3):271--282, 1984.


\bibitem{DKW}
\newblock M. Del Pino, M. Kowalczyk and J. Wei. 
\newblock \emph {On De Giorgi's conjecture in dimension $N\ge 9$}.
\newblock Ann. of Math. (2), 174, 3 : 1485-1569, (2011).



\bibitem{De}
\newblock F.~Demengel.
\newblock Qualitative properties of a non-linear system involving the p-Laplacian operator and a De Giorgi type result.
\newblock {\em Appl. Anal.} 94(4):725--746, 2015.



\bibitem{Di}
\newblock S.~Dipierro.
\newblock Geometric inequalities and symmetry results for elliptic systems,
\newblock {\em Discrete Contin. Dyn. Syst. A}, 33(8):3473--3496, 2013.

\bibitem{DiPi}
\newblock S.~Dipierro and A.~Pinamonti.
\newblock Symmetry results for stable and monotone solutions to fibered systems of PDEs.
\newblock {\em Commun. Contemp. Math.}, 17(4):1450035, 2015. 







\bibitem{DroMal}
\newblock N. Dror, B.~A. Malomed and J. Zeng.
\newblock Domain walls and vortices in linearly coupled systems.
\newblock {\em Phys. Rev. E}, 84(4): 046602, 2011.


\bibitem{FarFEQL} 
\newblock A. Farina,
\newblock Finite-energy solutions, quantization
effects and Liouville-type results for a variant of the Ginzburg-Landau
systems in $\R^K$.
\newblock {\em Differential Integral Equations}, 11(6):875-893, 1998.


%\bibitem{FaS}
%A.~Farina.
%\newblock Some symmetry results for entire solutions of an elliptic system
%  arising in phase separation.
%\newblock {\em Discrete Contin. Dyn. Syst.}, 34(6):2505--2511, 2014.





\bibitem{Fa} A.~Farina.
\newblock Rigidity and one-dimensional symmetry for semilinear elliptic equations in the whole of $\mathbb{R}^N$ and in half spaces.
\newblock {\em Adv. Math. Sci. Appl.}, 13(1), 2003,  65 -- 82.

\bibitem{FarSyst} A.~Farina.
\newblock Some symmetry results for entire solutions of an elliptic system arising in phase separation.
\newblock {\em Discrete Contin. Dyn. Syst. A}, 34(6): 2505--2511, 2014.

\bibitem{Fa99}
A.~Farina.
\newblock Symmetry for solutions of semilinear elliptic equations in {$\mathbf{R}^N$} and related conjectures.
\newblock {\em Ricerche Mat.}, 48(suppl.):129--154, 1999.
\newblock Papers in memory of Ennio De Giorgi (Italian).


%\bibitem{FarHand} A.~Farina.
%\newblock Liouville-type Theorems for elliptic problems,
%\newblock {\em Ch. 2, pp.61-116, in Handbook of Differential Equations: Stationary Partial %Differential Equations. Vol. 4, 2007,
%Ed. by M.Chipot, Elsevier B.V.}


\bibitem {FMS} A.~Farina, L.~Montoro and B.~Sciunzi,
{\newblock  Monotonicity and one-dimensional symmetry for solutions of  $ - \Delta_p u = f(u)$ in half-spaces.}
\newblock {\em Calc. Var. Partial Differential Equations}, 43(1-2), pp. 123--145, 2012.

\bibitem{FaSo}
A.~Farina and N.~Soave.
\newblock Monotonicity and 1-dimensional symmetry for solutions of an elliptic
  system arising in {B}ose-{E}instein condensation.
\newblock {\em Arch. Ration. Mech. Anal.}, 213(1):287--326, 2014.

\bibitem{FaVaInd}
A.~Farina and E.~Valdinoci.
\newblock Rigidity results for elliptic {PDE}s with uniform limits: an abstract
  framework with applications.
\newblock {\em Indiana Univ. Math. J.}, 60(1):121--141, 2011.



\bibitem{FaVaTAMS}
\newblock A.~Farina and E.~Valdinoci.
\newblock 1D symmetry for solutions of semilinear and quasilinear elliptic equations. 
\newblock {\em Trans. Amer. Math. Soc.} 363(2):579--609, 2011.



\bibitem{Faz}
\newblock M.~Fazly.
\newblock Entire solutions of quasilinear symmetric systems.
\newblock Preprint arXiv:1506.02731, 2015.

\bibitem{FaGh}
\newblock M.~Fazly and N.~Ghoussoub.
\newblock De Giorgi type results for elliptic systems.
\newblock {\em Calc. Var. Partial Differential
Equations} 47:809-823, 2013.


\bibitem{Fu}
\newblock G.~Fusco.
\newblock Layered solutions to the vector Allen-Cahn equation in $\R^2$. Characterization of minimizers and a new approach to heteroclinic connections.
\newblock Preprint arXiv: 1609.05306, 2016.




\bibitem{GG}
\newblock N. Ghoussoub, C. Gui. 
\newblock \emph {On a conjecture of De Giorgi and some related problems.}
\newblock Math. Ann. 311, 481-491 (1998). 



\bibitem {Kato}
\newblock T. Kato,
\newblock \emph {Schr\"{o}dinger operators with singular potentials},
\newblock Proceedings of the International Symposium on
Partial Differential Equations and the Geometry of Normed Linear Spaces (Jerusalem, 1972), Israel J. Math.
13 (1972) (1973), 135-148.

\bibitem{Mal}
\newblock B.~A.~Malomed, A.~A.~Nepomnyashchy and M.~I.~Tribelsky.
\newblock Domain boundaries in convection patterns.
\newblock {\em Phys Rev A.}, 42(12):7244--7263, 1990.


\bibitem{NoTaTeVe}
\newblock B.~Noris, H.~Tavares, S.~Terracini, and G.~Verzini.
\newblock Uniform {H}\"older bounds for nonlinear {S}chr\"odinger systems with
  strong competition.
\newblock {\em Comm. Pure Appl. Math.}, 63(3):267--302, 2010.

\bibitem{Sa}
\newblock O.~Savin.
\newblock Regularity of flat level sets in phase transitions.
\newblock {\em Ann. of Math. (2)} 169(1):41--78, 2009.

\bibitem{Ser}
\newblock J.~Serrin.
\newblock \emph {Nonlinear elliptic equations of second order},
\newblock Lectures at AMS Symposium on
Partial Differential Equations, Berkeley, 1971.




\bibitem{SoTe}
N.~Soave and S.~Terracini.
\newblock Liouville theorems and 1-dimensional symmetry for solutions of an
  elliptic system modelling phase separation.
\newblock {\em Adv. Math.}, 279:29--66, 2015.


\bibitem{SoZi1}
N.~Soave and A.~Zilio.
\newblock Entire solutions with exponential growth for an elliptic system
  modelling phase separation.
\newblock {\em Nonlinearity}, 27(2):305--342, 2014.

\bibitem{SoZi2}
N.~Soave and A.~Zilio.
\newblock Multidimensional entire solutions for an elliptic system
  modelling phase separation.
\newblock {\em Anal. PDE}, 9(5):1019--1041, 2016.

\bibitem{SoZiP}
N.~Soave and A.~Zilio.
\newblock On phase separation in systems of coupled elliptic equations:
  asymptotic analysis, geometric aspects and entire solutions.
\newblock  {\em Ann. Inst. H. Poincar� (C) Anal. Non Lin\'eaire}, in press. Doi: 10.1016/j.anihpc.2016.04.001.



\bibitem{Wa2}
K.~Wang.
\newblock Harmonic approximation and improvement of flatness in a singular
  perturbation problem.
\newblock {\em Manuscripta Math.}, 146(1-2):281--298, 2015.


\bibitem{Wa1}
K.~Wang.
\newblock On the {D}e {G}iorgi type conjecture for an elliptic system modeling
  phase separation.
\newblock {\em Comm. Partial Differential Equations}, 39(4):696--739, 2014.


\end{thebibliography}
\end{document}